\documentclass[english]{amsart}
\usepackage[T1]{fontenc}
\usepackage[latin9]{inputenc}
\usepackage{amstext}
\usepackage{amsthm}
\usepackage{amssymb}

\makeatletter
\numberwithin{equation}{section}
\numberwithin{figure}{section}
\theoremstyle{plain}
\newtheorem{thm}{\protect\theoremname}
\theoremstyle{remark}
\newtheorem{rem}[thm]{\protect\remarkname}
\theoremstyle{definition}
\newtheorem{example}[thm]{\protect\examplename}

\makeatother

\usepackage{babel}
\providecommand{\examplename}{Example}
\providecommand{\remarkname}{Remark}
\providecommand{\theoremname}{Theorem}

\begin{document}
\title{On Certain Bounds for Genus $0$ Entire Functions }
\author{Ruiming Zhang}
\email{ruimingzhang@guet.edu.cn}
\address{School of Mathematics and Computing Sciences\\
Guilin University of Electronic Technology\\
Guilin, Guangxi 541004, P. R. China. }
\begin{abstract}
In this work we use an elementary method to derive an upper bound
on the right half-plane for genus $0$ entire functions if it has
only negative zeros. The bound only uses information of the function
on the positive real axis. Applications to the modified Bessel functions
and the Airy function are also provided.
\end{abstract}

\maketitle

\section{Introduction}

It is well-known that many transcendental special functions are entire
functions of order at most $1$ with only real zeros. This class of
entire functions are closely related to genus $0$ entire functions
with only negative zeros, \cite{AndrewsAskeyRoy,DLMF,Ismail,WhattakerWatson}.
In this work we use elementary method to derive an upper bound for
genus $0$ entire functions on the right-half plane in terms of their
properties on the positive real line, and apply the bound to the modified
Bessel functions and the Airy function.

Let \cite{Boas}
\begin{equation}
f(z)=\sum_{n=0}^{\infty}a_{n}z^{n},\quad z\in\mathbb{C}\label{eq:1.1}
\end{equation}
be an entire function, then its order $\rho$ can be found by 

\begin{equation}
{\displaystyle \rho=\limsup_{r\to\infty}\frac{\log\left(\log\left(\max\left\{ \left|f(z)\right|:|z|\le r\right\} \right)\right)}{\log(r)}}\label{eq:1.2}
\end{equation}
or
\begin{equation}
\rho=\limsup_{n\to\infty}\frac{n\log n}{-\log|a_{n}|}.\label{eq:1.3}
\end{equation}
For any entire function $f(z)$ of finite order it is well-known that
it has the Hadamard's canonical representation,

\begin{equation}
{\displaystyle f(z)=z^{m}e^{P(z)}\prod_{n=1}^{\infty}\left(1-\frac{z}{z_{n}}\right)\exp\left(\frac{z}{z_{n}}+\cdots+\frac{1}{p}\left(\frac{z}{z_{n}}\right)^{p}\right),}\label{eq:1.4}
\end{equation}
where ${\displaystyle z_{k}}$ are nonzero roots of ${\displaystyle f}$,
${\displaystyle m}$ the multiplicity of root $0$, ${\displaystyle P}$
a polynomial with degree $q$, and ${\displaystyle p}$ the smallest
nonnegative integer such that the series

\begin{equation}
{\displaystyle \sum_{n=1}^{\infty}\frac{1}{|z_{n}|^{p+1}}<\infty.}\label{eq:1.5}
\end{equation}
The genus of $f(z)$ is defined as ${\displaystyle g=\max\{p,q\}}$.
If the order $\rho$ is not an integer, then ${\displaystyle g=\left\lfloor \rho\right\rfloor }$,
if $\rho\in\mathbb{N}$, then we may have ${\displaystyle g=\rho-1}$
or ${\displaystyle g=\rho}$. In particular, if $f(z)$ is of order
$\rho<1$, then it has genus $0$.

Clearly, an even entire function $g(z)$ has only real zeros if and
only if $f(z)=g(i\sqrt{z})$ has only negative zeros. By (\ref{eq:1.2})
it is evident that if $g(z)$ has a finite order, then $g(z)g(-z)$
is an even entire function with the same order. If $g(z)$ is an even
entire function of order $\rho<\infty$, by (\ref{eq:1.3}) $f(z)=g(i\sqrt{z})$
is an entire function of order $\frac{\rho}{2}$. 

\section{Main Result}
\begin{thm}
\label{thm:1}Let $f(z)$ be an entire function of order $\rho_{0}$
with $0\le\rho_{0}<1$, and assume that it has only negative zeros.
Then for any $\rho$ and $z\in\mathbb{C}$ with $\rho_{0}<\rho<1,\left|\arg z\right|<\frac{\pi}{2}$,
\begin{equation}
1\le\left|\frac{f\left(\Re(z)\right)}{f(0)}\right|\le\left|\frac{f(z)}{f(0)}\right|\le\exp\left(\frac{(\rho/e)^{\rho}\left|z\right|^{\rho}}{\cos^{1-\rho}(\arg z)\Gamma(1+\rho)}\int_{0}^{\infty}\frac{f'(x)}{f(x)}\frac{dx}{x^{\rho}}\right).\label{eq:2.1}
\end{equation}
\end{thm}

\begin{proof}
Since $f(z)$ has order $\rho_{0}<1$, then it is of genus $0$. Since
it has only negative zeros $\left\{ -z_{n}\right\} _{n\in\mathbb{N}}$,
\begin{equation}
0<z_{1}\le z_{2}\le\dots,\label{eq:2.2}
\end{equation}
then 
\begin{equation}
\frac{f(z)}{f(0)}=\prod_{n=1}^{\infty}\left(1+\frac{z}{z_{n}}\right).\label{eq:2.3}
\end{equation}
By \cite[Theorem 2.5.18]{Boas} we also have
\begin{equation}
\sum_{n=1}^{\infty}\frac{1}{z_{n}^{\rho}}<\infty,\quad\forall\rho_{0}<\rho.\label{eq:2.4}
\end{equation}
Then for $z=x+iy$ with $x\ge0$,
\begin{equation}
\begin{aligned} & \left|\frac{f(z)}{f(0)}\right|^{2}=\prod_{n=1}^{\infty}\left|1+\frac{z}{z_{n}}\right|^{2}=\prod_{n=1}^{\infty}\left(\left(1+\frac{x}{z_{n}}\right)^{2}+\frac{y^{2}}{z_{n}^{2}}\right)\\
 & \ge\prod_{n=1}^{\infty}\left(1+\frac{x}{z_{n}}\right)^{2}=\left|\frac{f(x)}{f(0)}\right|^{2}\ge1,
\end{aligned}
\label{eq:2.5}
\end{equation}
which shows 
\begin{equation}
\left|\frac{f(z)}{f(0)}\right|\ge\left|\frac{f(x)}{f(0)}\right|\ge1,\quad x=\Re z\ge0.\label{eq:2.6}
\end{equation}

Let 
\begin{equation}
\varphi(t)=\sum_{n=1}^{\infty}e^{-z_{n}t},\quad t>0,\label{eq:2.7}
\end{equation}
then for any $\rho>\rho_{0}\ge0$,
\begin{equation}
\int_{0}^{\infty}\varphi(t)t^{\rho-1}dt=\sum_{n=1}^{\infty}\int_{0}^{\infty}e^{-z_{n}t}t^{\rho-1}dt=\sum_{n=1}^{\infty}\frac{\Gamma(\rho)}{z_{n}^{\rho}}<\infty,\label{eq:2.8}
\end{equation}
where $\Gamma(z)$ is the Euler Gamma function, \cite[(5.2.1)]{DLMF}. 

On the other hand, since
\begin{equation}
t^{\rho}\varphi(t)=\sum_{n=1}^{\infty}\frac{1}{z_{n}^{\rho}}\left((tz_{n})^{\rho}e^{-z_{n}t}\right)\le\sup_{t>0}\left(t^{\rho}e^{-t}\right)\sum_{n=1}^{\infty}\frac{1}{z_{n}^{\rho}}=\left(\frac{\rho}{e}\right)^{\rho}\sum_{n=1}^{\infty}\frac{1}{z_{n}^{\rho}},\label{eq:2.9}
\end{equation}
then
\begin{equation}
\sup_{t>0}t^{\rho}\varphi(t)\le\left(\frac{\rho}{e}\right)^{\rho}\sum_{n=1}^{\infty}\frac{1}{z_{n}^{\rho}}.\label{eq:2.10}
\end{equation}

For any $x>0$, it is clear that 
\begin{equation}
\frac{f'(x)}{f(x)}=\sum_{n=1}^{\infty}\frac{1}{x+z_{n}}=\int_{0}^{\infty}e^{-xt}\varphi(t)dt,\label{eq:2.11}
\end{equation}
then 
\begin{equation}
\begin{aligned} & \int_{0}^{\infty}\frac{f'(x)}{f(x)}x^{-\rho}dx=\int_{0}^{\infty}\left(\int_{0}^{\infty}e^{-xt}x^{-\rho}dx\right)\varphi(t)dt\\
 & =\Gamma(1-\rho)\int_{0}^{\infty}t^{\rho-1}\varphi(t)dt=\Gamma(1-\rho)\Gamma(\rho)\sum_{n=1}^{\infty}\frac{1}{z_{n}^{\rho}}<\infty.
\end{aligned}
\label{eq:2.12}
\end{equation}
Therefore,
\begin{equation}
\sup_{t>0}t^{\rho}\left|\varphi(t)\right|\le\frac{\left(\rho/e\right)^{\rho}}{\Gamma(1-\rho)\Gamma(\rho)}\int_{0}^{\infty}\frac{f'(x)}{f(x)}\frac{dx}{x^{\rho}}.\label{eq:2.13}
\end{equation}
Since for $x>0$ 
\begin{equation}
0<1-e^{-x}\le\int_{0}^{x}e^{-u}du<x,\label{eq:2.14}
\end{equation}
 then for any $z\in\mathbb{C}$ with $\left|\arg z\right|<\frac{\pi}{2}$,
\begin{equation}
1-e^{-z}=\int_{0}^{z}e^{-u}du=e^{i\arg z}\int_{0}^{\left|z\right|}e^{-ue^{i\arg z}}du\label{eq:2.15}
\end{equation}
and
\begin{equation}
\left|1-e^{-z}\right|\le\int_{0}^{\left|z\right|}e^{-u\cos(\arg z)}du=\frac{1-e^{-\Re(z)}}{\cos(\arg z)}.\label{eq:2.16}
\end{equation}
Integrate (\ref{eq:2.11}) from $0$ to $x>0$, by (\ref{eq:2.8}),
(\ref{eq:2.14}) and Fubini's theorem,
\begin{equation}
\log\frac{f(x)}{f(0)}=\int_{0}^{\infty}\frac{1-e^{-xt}}{t}\varphi(t)dt.\label{eq:2.17}
\end{equation}

On one hand, $\log\frac{f(z)}{f(0)}$ is analytic in $\Re(z)>0$ because
the entire function $f(z)$ has only negative zeros. 

On the other hand, by (\ref{eq:2.14}), for $z$ with $\Re(z)>0$,
\begin{equation}
\begin{aligned} & \left|\int_{0}^{\infty}\frac{1-e^{-zt}}{t}\varphi(t)dt\right|\le\int_{0}^{\infty}\left|\frac{1-e^{-zt}}{t}\right|\varphi(t)dt\\
 & \le\int_{0}^{\infty}\left(\frac{1-e^{-t\Re(z)}}{t\cos(\arg z)}\right)\varphi(t)dt\le\frac{\Re(z)}{\cos(\arg z)}\int_{0}^{\infty}\varphi(t)dt=\left|z\right|\sum_{n=1}^{\infty}\frac{1}{z_{n}}<\infty,
\end{aligned}
\label{eq:2.18}
\end{equation}
then $\int_{0}^{\infty}\frac{1-e^{-zt}}{t}\varphi(t)dt$ converges
uniformly on any compact set of $\Re(z)>0$, thus it defines an analytic
function on the right half-plane. Since the two analytic functions
$\log\frac{f(z)}{f(0)}$ and $\int_{0}^{\infty}\frac{1-e^{-zt}}{t}\varphi(t)dt$
agree for $x>0$ through (\ref{eq:2.17}), then we must have

\begin{equation}
\log\frac{f(z)}{f(0)}=\int_{0}^{\infty}\frac{1-e^{-zt}}{t}\varphi(t)dt\label{eq:2.19}
\end{equation}
throughout $\Re(z)>0$ by analytic continuation. 

Since 
\[
\int_{0}^{\infty}\frac{1-e^{-t}}{t^{\rho+1}}dt=\frac{1}{\rho}\int_{0}^{\infty}e^{-t}t^{-\rho}dt=\frac{\Gamma(1-\rho)}{\rho},
\]
then for any $z\in\mathbb{C}$ with $\left|\arg z\right|<\frac{\pi}{2}$
by (\ref{eq:2.13}), 

\begin{equation}
\begin{aligned} & \left|\log\frac{f(z)}{f(0)}\right|=\left|\int_{0}^{\infty}\frac{1-e^{-zt}}{t}\varphi(t)dt\right|\\
 & \le\int_{0}^{\infty}\left|\frac{1-e^{-zt}}{t}\varphi(t)\right|dt\le\frac{1}{\cos(\arg z)}\int_{0}^{\infty}\frac{1-e^{-t\Re(z)}}{t}\varphi(t)dt\\
 & =\frac{1}{\cos(\arg z)}\int_{0}^{\infty}\frac{1-e^{-t\Re(z)}}{t^{1+\rho}}t^{\rho}\varphi(t)dt\le\frac{\sup_{t>0}t^{\rho}\varphi(t)}{\cos(\arg z)}\int_{0}^{\infty}\frac{1-e^{-t\Re(z)}}{t^{1+\rho}}dt\\
 & =\frac{\left(\Re(z)\right)^{\rho}\sup_{t>0}t^{\rho}\varphi(t)}{\cos(\arg z)}\int_{0}^{\infty}\frac{1-e^{-t}}{t^{1+\rho}}dt=\frac{\left(\Re(z)\right)^{\rho}\sup_{t>0}t^{\rho}\varphi(t)}{\cos(\arg z)}\frac{\Gamma(1-\rho)}{\rho}\\
 & \le\frac{(\rho/e)^{\rho}\left|z\right|^{\rho}}{\cos^{1-\rho}(\arg z)\Gamma(1+\rho)}\int_{0}^{\infty}\frac{f'(x)}{f(x)}\frac{dx}{x^{\rho}}.
\end{aligned}
\label{eq:2.20}
\end{equation}
 Therefore,
\begin{equation}
\log\left|\frac{f(z)}{f(0)}\right|\le\frac{(\rho/e)^{\rho}\left|z\right|^{\rho}}{\cos^{1-\rho}(\arg z)\Gamma(1+\rho)}\int_{0}^{\infty}\frac{f'(x)}{f(x)}\frac{dx}{x^{\rho}}.\label{eq:2.21}
\end{equation}
\end{proof}
\begin{rem}
There are many nice monotone properties associated with functions
$\frac{f(z)}{f(0)}$ defined in (\ref{eq:2.3}). For example, if $f(0)>0$,
then $\log\frac{f(x)}{f(0)}$ is a Bernstein function, consequently
both $\frac{f'(x)}{f(x)}$ and 
\begin{equation}
\frac{1}{f^{\beta}(x)}=\frac{1}{f^{\beta}(0)}\exp\left(-\beta\log\frac{f(x)}{f(0)}\right),\quad\beta>0\label{eq:2.22}
\end{equation}
are completely monotonic in $x\in(0,\infty)$. Therefore, they are
Laplace transforms of a finite nonnegative measure, \cite{Schilling}.
Then for $z=x+iy$ with $x>0$ we automatically have
\begin{equation}
\left|\left(\frac{1}{f^{\beta}(z)}\right)^{(n)}\right|\le(-1)^{n}\left(\frac{1}{f^{\beta}(x)}\right)^{(n)},\quad\beta>0,\ n\in\mathbb{N}_{0}\label{eq:2.23}
\end{equation}
 and
\begin{equation}
\left|\left(\frac{f'(z)}{f(z)}\right)^{(n)}\right|\le(-1)^{n}\left(\frac{f'(x)}{f(x)}\right)^{(n)},\quad n\in\mathbb{N}_{0}.\label{eq:2.24}
\end{equation}
 On the other hand, for any $x,y\in\mathbb{R}$ by the Jensen's theorem,
\cite{Gasper2}
\begin{equation}
\partial_{y}^{2}\left|f(x+iy)\right|^{2}\ge0,\quad y\partial_{y}\left|f(x+iy)\right|^{2}\ge0.\label{eq:2.25}
\end{equation}
\end{rem}

\section{Examples}
\begin{example}
For $\nu>-1$ the Bessel function of the first kind $J_{\nu}(z)$
is defined by \cite[(10.2.2)]{DLMF} \cite{AndrewsAskeyRoy,Ismail,WhattakerWatson}
\begin{equation}
\frac{J_{\nu}(z)}{(z/2)^{\nu}}=\sum_{k=0}^{\infty}\frac{(-1)^{k}(z/2)^{2k}}{k!\Gamma(\nu+k+1)}.\label{eq:3.1}
\end{equation}
The modified Bessel function of the first kind is given by \cite[(10.27.6)]{DLMF}
\begin{equation}
I_{\nu}(z)=i^{-\nu}J_{\nu}(iz).\label{eq:3.2}
\end{equation}
$\frac{J_{\nu}(z)}{(z/2)^{\nu}}$ is an even entire function of order
$1$ with only real zeros. Then
\begin{equation}
f(z)=\frac{I_{\nu}(\sqrt{z})}{(\sqrt{z}/2)^{\nu}},\quad\nu>-1\label{eq:3.3}
\end{equation}
is an entire function of order $\frac{1}{2}$ with only negative zeros. 

For any $1>\rho>\frac{1}{2}$ and $\left|\arg z\right|<\frac{\pi}{2}$,
by (\ref{eq:2.1}) we get
\begin{equation}
1\le\left|\frac{I_{\nu}(\sqrt{z})}{(\sqrt{z}/2)^{\nu}}\right|\le\exp\left(a_{\rho}(\nu,z)\left|z\right|^{\rho}\right),\label{eq:3.4}
\end{equation}
 where
\begin{equation}
a_{\rho}(\nu,z)=\frac{(\rho/e)^{\rho}}{\cos^{1-\rho}(\arg z)\Gamma(1+\rho)}\int_{0}^{\infty}\left(\frac{I_{\nu-1}(\sqrt{x})}{2\sqrt{x}I_{\nu}(\sqrt{x}}-\frac{3\nu}{2x}\right)\frac{dx}{x^{\rho}}.\label{eq:3.5}
\end{equation}
\end{example}

\begin{example}
The Airy function $\text{Ai}(z)$ can be defined by \cite[(9.6.2)]{DLMF},\cite{WhattakerWatson}
\begin{equation}
\text{Ai}(z)=\frac{\sqrt{z}}{3}\left(I_{-1/3}\left(\frac{2}{3}z^{3/2}\right)-I_{1/3}\left(\frac{2}{3}z^{3/2}\right)\right).\label{eq:3.6}
\end{equation}
It is an entire function of order $\frac{3}{2}$ with only negative
zeros. Then $\text{Ai}(z)\text{Ai}(-z)$ is an even entire function
of order $\frac{3}{2}$ with only real zeros. Hence,
\begin{equation}
f(z)=\text{Ai}(i\sqrt{z})\text{Ai}(-i\sqrt{z})\label{eq:3.7}
\end{equation}
 is an entire function of order $\frac{3}{4}$ with only negative
zeros. Since by \cite[(9.2.3)]{DLMF} we have
\begin{equation}
f(0)=\text{Ai}^{2}(0)=\frac{1}{3^{4/3}\Gamma^{2}(2/3)},\label{eq:3.8}
\end{equation}
then for $1>\rho>\frac{3}{4}$ and $\left|\arg z\right|<\frac{\pi}{2}$,
by (\ref{eq:2.1}),
\begin{equation}
\frac{1}{3^{4/3}\Gamma^{2}(2/3)}\le\left|\text{Ai}(i\sqrt{z})\text{Ai}(-i\sqrt{z})\right|\le\frac{\exp\left(b_{\rho}(z)\left|z\right|^{\rho}\right)}{3^{4/3}\Gamma^{2}(2/3)},\label{eq:3.9}
\end{equation}
where
\begin{equation}
b_{\rho}(z)=\frac{2(\rho/e)^{\rho}}{\cos^{1-\rho}(\arg z)\Gamma(1+\rho)}\int_{0}^{\infty}\frac{d}{dx}\log\left(\left|\text{Ai}(i\sqrt{x})\right|\right)\frac{dx}{x^{\rho}}.\label{eq:3.10}
\end{equation}
\end{example}

\begin{example}
The modified Bessel function $K_{\nu}(z)$ is defined by \cite[(10.27.4)]{DLMF},\cite{AndrewsAskeyRoy,Gasper2,Ismail,WhattakerWatson}
\begin{equation}
K_{\nu}(z)=\frac{\pi}{2}\frac{I_{-\nu}(z)-I_{\nu}(z)}{\sin(\nu\pi)}.\label{eq:3.11}
\end{equation}
For any $z$ with $\left|\arg(z)\right|<\frac{\pi}{2}$ it has the
integral representation \cite[(10.32.9)]{DLMF}, 
\begin{align}
K_{\nu}(z) & =\frac{1}{2}\int_{-\infty}^{\infty}\exp\left(-z\cosh u-\nu u\right)du.\label{eq:3.12}
\end{align}
In particular, for any $a>0$, as a function of variable $z$, $K_{iz}\left(a\right)$
is an even order $1$ entire function that has only real zeros, \cite{Gasper2}.
Consequently, 
\begin{equation}
f(z)=K_{\sqrt{z}}\left(a\right)=K_{-\sqrt{z}}\left(a\right),\quad f(0)=K_{0}(a)>0.\label{eq:3.13}
\end{equation}
is an entire function of order $\frac{1}{2}$ with only negative zeros. 

Then by (\ref{eq:2.1}) for $\frac{1}{2}<\rho<1$ and $z$ with $\left|\arg z\right|<\frac{\pi}{2}$
we have
\begin{equation}
1\le\left|\frac{K_{\sqrt{z}}\left(a\right)}{K_{0}\left(a\right)}\right|\le\exp\left(c_{\rho}(z,a)\left|z\right|^{\rho}\right),\label{eq:3.14}
\end{equation}
where
\begin{equation}
c_{\rho}(z,a)=\frac{(\rho/e)^{\rho}}{\Gamma(1+\rho)\cos^{1-\rho}(\arg z)}\int_{0}^{\infty}\frac{d}{dx}\left(\log K_{\sqrt{x}}\left(a\right)\right)\frac{dx}{x^{\rho}}\label{eq:3.15}
\end{equation}
\end{example}


\begin{thebibliography}{1}
\bibitem{Gasper2} G. Gasper, Using integrals of squares to prove
that certain real-valued special functions to prove that the P\'{o}lya
$\Xi^{*}\left(z\right)$ function, the functions $K_{iz}\left(a\right),\ a>0$
and some other entire functions have only real zeros, dedicated to
Dan Waterman on the occasion of his 80th birthday.
\end{thebibliography}
\end{document}